[1994/12/01]
\documentclass[draft]{article}
\title{\textbf{A note on the lower bounds of the first nonzero Steklov eigenvalue on compact manifolds}}
\author{Yiwei Liu and Yi-Hu Yang{\footnote{Partially supported by NSF of China (No. 12071283)}}}

\date{\today}

\usepackage{amsmath,amsthm,amssymb}
\usepackage[all,pdf]{xy}
\usepackage[numbers,sort&compress]{natbib}
\chardef\bslash=`\\ 

\newtheorem{theorem}{\bf Theorem}[section]
\newtheorem{lemma}[theorem]{\bf Lemma}
\newtheorem{proposition}[theorem]{\bf Proposition}

\newtheorem{corollary}[theorem]{\bf Corollary}
\newtheorem{example}[theorem]{\bf Example}
\newtheorem{defn}[theorem]{\bf Definition}

\newenvironment{proof1}{\noindent{\em Proof of Theorem  \ref{thm: 1.1}:}}{\quad \hfill$\Box$\vspace{2ex}}
\newenvironment{proof2}{\noindent{\em Proof of Theorem  \ref{thm: 1.3}:}}{\quad \hfill$\Box$\vspace{2ex}}
\newenvironment{proof3}{\noindent{\em Proof of Theorem  \ref{thm: 1.5}:}}{\quad \hfill$\Box$\vspace{2ex}}
\newenvironment{proof4}{\noindent{\em Proof of Corollary  \ref{Cor: 1.2}:}}{\quad \hfill$\Box$\vspace{2ex}}

\newenvironment{remark}{\noindent{\bf Remark.}}{\vspace{2ex}}

\newcommand{\eval}[2][\right]{\relax
  \ifx#1\right\relax \left.\fi#2#1\rvert}

\begin{document}
\maketitle

\renewcommand{\sectionmark}[1]{}

\begin{abstract}
Let $(\Omega^{n+1}, g)$ be an $(n+1)$-dimensional smooth compact connected Riemannian manifold 
with smooth boundary $\Sigma$, satisfying that ${\text{Ric}_{\Omega}}\ge 0$ and $\Sigma$ is strictly convex,
more precisely, its second fundamental form $h\ge cg_{\Sigma}$ for some positive constant $c$. 
Escobar {\cite{escobar1997geometry}} considered the first 
nonzero Steklov eigenvalue $\sigma_1$ of $(\Omega^{n+1}, g)$ 
and proved that $\sigma_1\geq c$ 
when $n=1$ and $\sigma_1>{\frac{c}{2}}$ 
when $n \geq 2$. He then conjectured 
{\cite{escobar1999isoperimetric}} that the first nonzero Steklov
eigenvalue $\sigma_1\ge c$. Very recently, Xia and Xiong {\cite{xia2023escobar}} 
confirmed Escobar's conjecture in the case that $\Omega$ has nonnegative sectional curvature,
by constructing a weight function and using appropriate integral identities.
In this paper, we construct a new weight function under 
certain sectional curvature assumptions 
and provide some new lower bounds for the first nonzero Steklov eigenvalue, which can be considered
as generalizations of the results of Escobar and Xia-Xiong. As an application of 
the weight function, we also consider lower bound estimate 
of the first nonzero Steklov eigenvalue under conformal transformations.
\end{abstract}

\section{Introduction}

Let $(\Omega^{n+1},g)$ be an $(n + 1)$-dimensional $(n\geq1)$ smooth compact connected Riemannian manifold with smooth boundary $\partial\Omega=\Sigma$.  
The Steklov eigenvalue problem is the following (see \cite{kuznetsov2014legacy} and \cite{stekloff1902problemes})
\begin{equation}
\begin{cases}
\begin{aligned}
    	&\Delta u=0, & \mbox{in} \,\Omega,\\
    	&\frac{\partial u}{\partial \nu}=\sigma u, & \mbox{on} \, \Sigma,
\end{aligned}
\end{cases}
\label{e:1}
\end{equation}
where $\Delta$ is the Laplace-Beltrami operator of $\Omega$ and $\nu$ is the outward unit normal vector on $\Sigma$. It is well-known that the spectrum of the eigenvalue problem (\ref{e:1}) is nonnegative, discrete and unbounded:
$$0=\sigma_{0}<\sigma_{1}\leq\sigma_{2}\leq\cdots\rightarrow+\infty.$$
For the first nonzero Steklov eigenvalue $\sigma_1$, we also know that
\begin{equation}
    \sigma_1=\inf_{u\in C^{\infty}(\Sigma), \int_{\Sigma}uda=0} \frac{\int_{\Omega} |\nabla (\mathcal{H}u)|^2 dA}{\int_{\Sigma} u^2 da},
\end{equation}
where $\mathcal{H}u$ is the harmonic extension of $u$ in $\Omega$.
For more information about the Steklov eigenvalue problem, interested
readers also can refer to \cite{girouard2017spectral}.

In this paper, we provide some results related to the first nonzero Steklov eigenvalue $\sigma_{1}$ by constructing a new weight function and using appropriate integral identities.

\subsection{Lower bound estimate of the first nonzero Steklov eigenvalue related to Escobar's conjecture}

In \cite{payne1970some}, by the maximum principle, Payne proved that for a bounded domain $\Omega \subset \mathbb{R}^2$, if the geodesic curvature $k_g$ of the boundary curve satisfies $k_g \geq c>0$, then the first nonzero Steklov eigenvalue $\sigma_1$ of $\Omega$ satisfies $\sigma_1 \geq c$, with equality holding if and only if $\Omega$ is a round disk of radius $\frac{1}{c}$. Later, Escobar \cite{escobar1997geometry} generalized Payne's result to $2$-dimensional compact manifolds with nonnegative Gaussian curvature by a similar method.
In higher dimensions, by using Reilly's formula (see \cite{reilly1977applications}), he \cite{escobar1997geometry} also provided a non-sharp estimate $\sigma_1 > \frac{c}{2}$ for compact manifolds $(\Omega^{n+1},g)$ which satisfy $\mbox{Ric}_{\Omega} \geq 0$ and $h \geq cg_{\Sigma} >0$, where $h \geq cg_{\Sigma} >0$ means the principal curvatures of $\Sigma$ $\geq c>0$.

Based on the above results, Escobar made the following conjecture (see \cite{escobar1999isoperimetric}).

\vskip .2cm
\noindent
{\bf Escobar's conjecture:} \emph{Let $(\Omega^{n+1},g)$ be an $(n+1)$-dimensional smooth compact connected Riemannian manifold with smooth boundary $\Sigma$. Assume that $\mbox{Ric}_{\Omega} \geq 0$ and $h \geq cg_{\Sigma} >0$. Then the first nonzero Steklov eigenvalue $\sigma_1$ satisfies
$$
\sigma_1 \geq c.
$$
Moreover, the equality holds if and only if $\Omega$ is isometric to a Euclidean ball
of radius $\frac{1}{c}$.}\\[3pt]

When $\Omega$ is a ball equipped with rotationally invariant metric, this conjecture had been proven by Montaño \cite{montano2013stekloff} (see also \cite{xiong2022spectra}). We also notice that Montaño \cite{montano2016escobar} confirmed this conjecture for Euclidean ellipsoids.
Later, Xia and Xiong \cite{xia2023escobar} showed that Escobar's conjecture is true for manifolds with nonnegative sectional curvature. It should be pointed out that their method is mainly based on the weighted Reilly-type formula (see \cite{qiu2015generalization}), the Pohozaev-type identity (see \cite{xiong2018comparison}) and a special weight function
$$V=\rho-\frac{c}{2}\rho^2,$$
where $\rho=d(\cdot,\Sigma)$ is the distance function to the boundary $\Sigma$. Interested readers
can also refer to \cite{duncan2024first} and \cite{LiuSteklov} for more information about $\sigma_1$.

In this paper, we first provide some results related to Escobar's conjecture.
For convenience, we introduce the necessary notions.

Let $(\Omega^{n+1}, g)$ be an $(n+1)$-dimensional ($n\geq 2$) smooth compact connected Riemannian manifold with smooth boundary $\Sigma$ and $\rho=d(\cdot,\Sigma)$ be the distance function to $\Sigma$. We now recall the concept of cut point. For $p \in \Sigma$, we consider the geodesic $\gamma_{p}(t)=\exp_{p}(-t\nu(p))$ with arc length parameter. Then we know that the point $\gamma(t_0)$ is a cut point of $\Sigma$ if 
$$t_0=\sup\{t>0~|~\rho(\gamma_{p}(t))=t\}.$$
Let $\mbox{Cut}(\Sigma)$ be the set of all cut points of $\Sigma$. We then know that the set $\mbox{Cut}(\Sigma)$ has zero $(n+1)$-dimensional Hausdorff measure and the function $\rho$ is smooth on $\Omega \setminus \mbox{Cut}(\Sigma)$. In addition, $\forall ~x \in \Omega \setminus \mbox{Cut}(\Sigma)$, let $\gamma:[0,\rho(x)] \rightarrow \Omega$ be the minimizing geodesic with arc length parameter such that $t = \rho(\gamma(t))$ for $t \in [0,\rho(x)]$, where $\gamma(0)\in \Sigma$ and $\gamma(\rho(x))=x$, we know that this geodesic is unique. Base on this unique geodesic, we define $K(x)$ as
$$K(x)=\inf\{\mbox{Sec}_{\Omega}(\gamma^{\prime}(\rho(x))\wedge X)~|~X\in T_{x}\Omega, \,|X|=1,\, X \perp \gamma^{\prime}(\rho(x))\},$$
where $\mbox{Sec}_{\Omega}$ is the sectional curvature of $\Omega$. We call $K(x)$ the infimum of the radial sectional curvature at $x$. Let 
$$K = \inf\{K(x)~|~x \in \Omega \setminus \mbox{Cut}(\Sigma)\}.$$

Now, we can state our results. By constructing a new weight function $V$, using appropriate integral identities and combining Escobar's result, we first provide a general lower bound of $\sigma_{1}$ in terms of the infimum $K$ and a constant $c>0$ for $(n+1)$-dimensional $(n\geq2)$ compact manifolds which satisfy $\mbox{Ric}_{\Omega} \geq 0$ and $h \geq cg_{\Sigma} >0$. 

\begin{theorem}
Let $(\Omega^{n+1},g)$ be an $(n+1)$-dimensional $(n\geq2)$ smooth compact connected Riemannian manifold with smooth boundary $\Sigma$. Assume that $\mbox{Ric}_{\Omega} \geq 0$ and $h \geq cg_{\Sigma} >0$. Then the first nonzero Steklov eigenvalue $\sigma_1$ satisfies
$$
    \sigma_1 \geq\sigma(c,K),
    $$
where 
$$
    \sigma(c,K)=
    \left\{
    \begin{aligned}
    	&c,~~&K\geq0,\\
    	&c\cosh{\bigg(\frac{\sqrt{|K|}}{c}\bigg)}-\sqrt{|K|}\sinh{\bigg(\frac{\sqrt{|K|}}{c}\bigg)},~~&K_0\leq K \leq 0, \\
    	&\frac{c}{2},~~~&K\leq K_0,
    \end{aligned}
    \right.
    $$
and
$K_0 \in (-\infty,0)$ is the unique number which satisfies 
$$c\cosh{\bigg(\frac{\sqrt{|K_0|}}{c}\bigg)}-\sqrt{|K_0|}\sinh{\bigg(\frac{\sqrt{|K_0|}}{c}\bigg)}=\frac{c}{2}.$$ 
Moreover, the equality holds if and only if $\Omega$ is isometric to a Euclidean ball
of radius $\frac{1}{c}$ and in this case, $K=0$, $\sigma_1=\sigma(c,K)=c$.
\label{thm: 1.1}
\end{theorem}

\begin{remark}
(1) We point out that when $K\geq0$, Theorem \ref{thm: 1.1} is actually Xia and Xiong's result in \cite{xia2023escobar}; when $K\leq K_0$, Theorem \ref{thm: 1.1} is just Escobar's result in \cite{escobar1997geometry}. 

(2) We also point out that 
$$\lim_{K \rightarrow 0} \bigg[c\cosh{\bigg(\frac{\sqrt{|K|}}{c}\bigg)}-\sqrt{|K|}\sinh{\bigg(\frac{\sqrt{|K|}}{c}\bigg)}\bigg]=c,$$
$$\lim_{K \rightarrow -\infty} \bigg[c\cosh{\bigg(\frac{\sqrt{|K|}}{c}\bigg)}-\sqrt{|K|}\sinh{\bigg(\frac{\sqrt{|K|}}{c}\bigg)}\bigg]=-\infty$$
and the function $t\rightarrow c\cosh{(\frac{t}{c})-t}\sinh{(\frac{t}{c})}$ is strictly monotonically decreasing when $t>0$, so the number $K_0$ exists uniquely and our estimate can also be seen as an improvement on Escobar's result.

(3) In \cite{LiuSteklov}, the present authors actually showed that when $K <0$, 
$$
\sigma_1 >
    \left\{
    \begin{aligned}
    	&c+\frac{K}{c},~~&-\frac{c^2}{2}\leq K < 0, \\
    	&\frac{c}{2},~~~&K\leq -\frac{c^2}{2}.
    \end{aligned}
    \right.
$$
On the other hand, using the property that the function
$c\cosh{(\frac{t}{c})-t}\sinh{(\frac{t}{c})}$ in $t>0$ is strictly monotonically decreasing, 
a direct calculation shows that $K_0<-\frac{c^2}{2}$ and when $K\in [-\frac{c^2}{2},0)$,
$$
    c\cosh{\bigg(\frac{\sqrt{|K|}}{c}\bigg)}-\sqrt{|K|}\sinh{\bigg(\frac{\sqrt{|K|}}{c}\bigg)} > c+\frac{K}{c}.
    \label{ineq:a}
$$
Therefore, Theorem \ref{thm: 1.1} actually covers the result in \cite{LiuSteklov}. 
\end{remark}

Based on the conclusion of Theorem \ref{thm: 1.1}, we have the following rigidity result.

\begin{corollary}
Let $(\Omega^{n+1},g)$ be as in Theorem \ref{thm: 1.1} with $K \geq K_0$. Suppose there exists a non-constant harmonic function $f \in C^{\infty}(\Omega)$ and a positive continuous function $k \leq \sigma(c,K)$ on $\Sigma$ which satisfies 
$$\frac{\partial f}{\partial \nu}=kf.$$
Then $\Omega$ is isometric to a Euclidean ball
of radius $\frac{1}{c}$ and the function $k$ must be $c$.
\label{Cor: 1.2}
\end{corollary}

The proof of Theorem \ref{thm: 1.1} is mainly based on the integral identities and a more general weight function $V$. Here we briefly introduce our idea. Let $(\Omega^{n+1}, g)$ be an $(n+1)$-dimensional ($n\geq 2$) smooth compact connected Riemannian manifold with smooth boundary $\Sigma$ and $\rho=d(\cdot,\Sigma)$ be the distance function to $\Sigma$. Let
$$\rho_{\max}=\max_{\Omega}\rho$$
and 
$f \in C([0,\rho_{\max}])$ be a continuous function. Suppose the Ricci curvature $\mbox{Ric}_{\Omega} \geq 0$, the infimum of the radial sectional curvature $K(x)\geq f(\rho(x))$, $\forall\, x \in \Omega \setminus \mbox{Cut}(\Sigma)$, and the second fundamental form $h \geq cg_{\Sigma} >0$. Let $\theta \in C^2 ([0,\rho_{\max}])$ be the unique solution of the following system
\begin{equation*}
\begin{cases}
\begin{aligned}
    	&\theta^{\prime \prime}+F\theta=0,\\
    	&\theta(0)=1,\\
        &\theta^{\prime}(0)=-c,
    \end{aligned}
\end{cases}
\end{equation*}
where 
$$F(t)=\min\{0,\min_{x\in [0,t]}f(x)\}$$
and $V(t)=\int_{0}^{t} \theta(s)ds$. Then the weight function (denote also by $V$) is given by
$$V(x)=V(\rho(x))=\int_{0}^{\rho(x)} \theta(s)ds.$$
It is elementary for us to show that $V \in C^3(\Omega \setminus \mbox{Cut}(\Sigma))$. Based on the curvature assumptions and the Hessian comparison theorem for $\rho$ (see Theorem 2.31 in \cite{kasue1982laplacian}), we can show that 
\begin{equation*}
\begin{aligned}
    	&\nabla^2(-V)|_x(X,X)\geq c+ (\widehat{F}V)(x)
    \end{aligned}
\end{equation*}
for $x\in \Omega \setminus \mbox{Cut}(\Sigma)$ and any unit $X \in T_x \Omega$, where $\widehat{F}=F\circ\rho$. 

It should be pointed out that the function $V$ can be regarded as a generalization of Xia and Xiong's weight function. In fact, if we assume that $f\equiv0$, it is elementary to show that 
$$V=\rho-\frac{c}{2}\rho^2.$$
In the proof of Theorem \ref{thm: 1.1}, by choosing $f\equiv K$, we can show that
$$
    V=
    \left\{
    \begin{aligned}
    	&\rho-\frac{c}{2}\rho^2,~~&K\geq0,\\
    	&\frac{1}{\sqrt{|K|}}\sinh(\sqrt{|K|}\rho)-\frac{c}
        {|K|}\cosh(\sqrt{|K|}\rho)+\frac{c}{|K|},~~& K < 0. 
    \end{aligned}
    \right.
    $$

To ensure that the weight function $V$ can be applied to the integral identities, we still need to consider a suitable approximation of $V$. By a similar argument of Xia and Xiong, we can show that there exists a Greene-Wu type approximation $V_{\epsilon} \in C^{3}(\Omega)$ of $V$ which satisfies $\nabla^2(-V_{\epsilon})\geq (c+ \widehat{F}V-\epsilon)g$ for any small $\epsilon>0$. We will present details about the weight function and its Greene-Wu type approximation in Section 3.

Here we also point out that if we strengthen the condition about the Ricci curvature, we can confirm Escobar's conjecture by the weight function and the integral identities. In fact, we have the following 

\begin{theorem}
Let $(\Omega^{n+1}, g)$ be an $(n+1)$-dimensional ($n\geq 2$) smooth compact connected Riemannian manifold with smooth boundary $\Sigma$ and $f \in C([0,\rho_{\max}])$ be a continuous function. 
Suppose 

$\bullet$ the Ricci curvature $\mbox{Ric}_{\Omega}\geq |\widehat{F}|g$,

$\bullet$ the infimum of the radial sectional curvature 
$K(x)\geq f(\rho(x))$, $\forall\, x \in \Omega \setminus \mbox{Cut}(\Sigma)$,

$\bullet$ the second fundamental form $h \geq cg_{\Sigma} >0$.

\noindent
Then the first nonzero Steklov eigenvalue $\sigma_1$ satisfies
$$
    \sigma_1 \geq c.
$$
Moreover, the equality holds if and only if $\Omega$ is isometric to a Euclidean ball
of radius $\frac{1}{c}$ and in this case $f\equiv0$.
\label{thm: 1.3}
\end{theorem}

\begin{remark}
Theorem \ref{thm: 1.3} can also be seen as a generalization of Xia and Xiong's result in \cite{xia2023escobar}. 
\end{remark}

 When $f$ is nonincreasing, we know that 
$$
F=\min\{0,f\}.
$$ 
By Theorem \ref{thm: 1.3}, we then have
\begin{corollary}
Let $(\Omega^{n+1}, g)$ be an $(n+1)$-dimensional ($n\geq 2$) smooth compact connected Riemannian manifold with smooth boundary $\Sigma$ and $f \in C([0,\rho_{\max}])$ be a nonincreasing continuous function. Suppose 

$\bullet$ the Ricci curvature $\mbox{Ric}_{\Omega}\geq |\min \{0,f\circ\rho\}|g$,

$\bullet$ the infimum of the radial sectional curvature 
$K(x)\geq f(\rho(x))$, $\forall\, x \in \Omega \setminus \mbox{Cut}(\Sigma)$,

$\bullet$ the second fundamental form $h \geq cg_{\Sigma} >0$.

\noindent
Then the first nonzero Steklov eigenvalue $\sigma_1$ satisfies
$$
    \sigma_1 \geq c.
$$
Moreover, the equality holds if and only if $\Omega$ is isometric to a Euclidean ball
of radius $\frac{1}{c}$ and in this case $f\equiv0$.
\label{Cor: 1.4}
\end{corollary}

\subsection{Lower bound estimate of the first nonzero Steklov eigenvalue 
under conformal transformations}

Based on the weight function introduced in the previous subsection, we can also 
provide a result on the first nonzero Steklov eigenvalue $\sigma_{1}$ under conformal 
transformations. We continue to use the constant $K$ defined in the previous subsection.

 Let $(\Omega^2,g)$ be a compact surface with smooth boundary $\Sigma$ and $\widehat{g}=e^{2f}g$ be a metric conformal to $g$, where $f \in C^{\infty}(\Omega)$ is a smooth function which satisfies $f|_{\Sigma}=0$. Denote by $\sigma_1(\widehat{g})$ the first nonzero Steklov eigenvalue with respect to the metric $\widehat{g}$. 
 
 In \cite{escobar1997geometry}, based on the fact that the Dirichlet integral is a conformal invariant, Escobar proved that if the Gaussian curvature $K_g$ is nonnegative and the geodesic curvature $k_{g}$ of the boundary satisfies $k_{g}\geq c>0$, then the first nonzero Steklov eigenvalue $\sigma_1(\widehat{g})$ satisfies $\sigma_1(\widehat{g})=\sigma_1 \geq c$. In fact, for any $u \in C^{\infty }(\Omega)$, we have
$$\frac{\int_{\Omega} |\nabla_{\widehat{g}} u|^2 dA_{\widehat{g}}}{\int_{\Sigma} u^2 da_{\widehat{g}}}=\frac{\int_{\Omega} |\nabla u|^2 dA}{\int_{\Sigma} u^2 da}.$$
The conclusion then follows from the maximum-minimum property and Escobar's conjecture for the two dimensional case.

Inspired by the result discussed above, we consider the lower bound estimate of the first nonzero Steklov eigenvalue under conformal transformations for manifolds with higher dimensions. Note that in this case the Escobar's method is ineffective, but fortunately, based on the weight function $V$, we have the following

\begin{theorem}
Let $(\Omega^{n+1}, g)$ be an $(n+1)$-dimensional ($n\geq 2$) smooth compact connected Riemannian manifold with smooth boundary $\Sigma$ which satisfies $\mbox{Ric}_{\Omega} \geq 0$ and $h \geq cg_{\Sigma} >0$. Let $f \in C^{\infty}(\Omega)$ be a smooth function which satisfies $f|_{\Sigma}=0$ and $\nabla^2 f \leq \frac{1}{n-1}\min\{0,K\}g$. Then the first nonzero Steklov eigenvalue $\sigma_1(\widehat{g})$ with respect to the metric $\widehat{g}=e^{2f}g$ satisfies
$$\sigma_1(\widehat{g})\geq c.$$
Moreover, the equality holds if and only if $(\Omega,g)$ is isometric to a Euclidean ball
of radius $\frac{1}{c}$ and $f\equiv0$.
\label{thm: 1.5}
\end{theorem}

For manifolds with $K\geq0$, we know that $\min\{0,K\}=0$.
Then by Theorem \ref{thm: 1.5}, we have

\begin{corollary}
Let $(\Omega^{n+1}, g)$ be an $(n+1)$-dimensional ($n\geq 2$) smooth compact connected Riemannian manifold with smooth boundary $\Sigma$ which satisfies $\mbox{Ric}_{\Omega}\geq0$, $K\geq 0$ and $h \geq cg_{\Sigma} >0$. Let $f \in C^{\infty}(\Omega)$ be a concave function which satisfies $f|_{\Sigma}=0$.  Then the first nonzero Steklov eigenvalue $\sigma_1(\widehat{g})$ with respect to the metric $\widehat{g}=e^{2f}g$ satisfies
$$\sigma_1(\widehat{g})\geq c.$$
Moreover, the equality holds if and only if $(\Omega,g)$ is isometric to a Euclidean ball
of radius $\frac{1}{c}$ and $f\equiv0$.
\label{Cor: 1.6}
\end{corollary}

We can also provide an example for the case that $\sigma_1(\widehat{g})>c$.
\begin{example}
    Let $\mathbb{B}^{n+1}$ be an $(n+1)$-dimensional unit ball in the Euclidean space $\mathbb{R}^{n+1}$. The boundary of $\mathbb{B}^{n+1}$ is an $n$-dimensional unit sphere $\mathbb{S}^n$. In this case, $c=1$. Then by Proposition 3.3 in \cite{xia2023escobar}, we know that there exists a smooth function $f$ on $\Omega$ which satisfies
    $$f|_{\mathbb{S}^n}=0$$
    and
    $$\nabla^2 f \leq -\frac{1}{2}g_{euc}<0,$$
    where $g_{euc}$ is the standard Euclidean metric.
    Let $\widehat{g}=e^{2f}g_{euc}$, we then know that 
    $$\sigma_{1}(\widehat{g})>1.$$
\end{example}

In addition to the weight function, the proof of Theorem \ref{thm: 1.5} also depends on the integral identities related to the Laplacian operator $\Delta_{\widehat{g}}$ with respect to the metric $\widehat{g}$. These identities are derived from the weighted Reilly-type formula, the Pohozaev-type identity and an observation that $\Delta_{\widehat{g}}=e^{-2f}\mathbb{L}_{-(n-1)f}$, where $\mathbb{L}_{-(n-1)f}$ is the weighted Laplacian operator which is defined by
$$\mathbb{L}_{-(n-1)f}=\Delta+(n-1)\langle\nabla f,\cdot\rangle.$$

Based on the fact that $\Delta_{\widehat{g}}=e^{-2f}\mathbb{L}_{-(n-1)f}$, we also point out that the first nonzero Steklov eigenvalue $\sigma_1(\widehat{g})$ is equal to the first nonzero Steklov-type eigenvalue $\tau_1$, where the Steklov-type eigenvalue problem is the following (see \cite{deng2020sharp})
\begin{equation*}
\begin{cases}
\begin{aligned}
    	&\mathbb{L}_{-(n-1)f} u=0, & \mbox{in} \,\Omega,\\
    	&\frac{\partial u}{\partial \nu}=\tau u, & \mbox{on} \, \Sigma.
\end{aligned}
\end{cases}
\end{equation*}
Our Theorem \ref{thm: 1.5} also shows that the first nonzero Steklov-type eigenvalue $\tau_1\geq c$.
Readers can refer to \cite{deng2020sharp} for the lower bound estimates of the first nonzero Steklov-type eigenvalue $\tau_1$.

The paper is organized as follows. In Section 2, we give some basic definitions 
and the integral identities which are needed later. Section 3 concentrates on the construction of the weight function and its Greene-Wu type approximation. In Section 4, we consider the lower bound of the first nonzero Steklov eigenvalue related to Escobar's conjecture and prove Theorems \ref{thm: 1.1} and \ref{thm: 1.3}. In Section 5, we consider the first nonzero Steklov eigenvalue under conformal transformations and provide the proof of Theorem \ref{thm: 1.5}.

\section{Preliminaries}

This section mainly introduces the necessary notions and the integral identities which are needed in the later proofs.

Let $(\Omega^{n+1},g)$ be an $(n + 1)$-dimensional smooth compact connected Riemannian manifold with smooth boundary $\partial \Omega=\Sigma$ and $g_{\Sigma}$ be
the induced metric on $\Sigma$; denote by $\langle \cdot, \cdot \rangle$ 
the inner product on $\Omega$ as well as $\Sigma$. Denote by
$\nabla^{\Omega}$, $\nabla$, $\Delta$, and $\nabla^2$ the connection, the gradient, the Laplacian, and the Hessian on $\Omega$ respectively, while by $\nabla_{\Sigma}$ 
and $\Delta_{\Sigma}$ the gradient and the Laplacian on $\Sigma$ respectively.
Let $\nu$ be the outward unit normal
vector on $\Sigma$. We denote by $h$
and $H$ the second fundamental form
and the mean curvature of $\Sigma$ with respect to $\nu$ respectively, where 
$$
h(X,Y)= -\langle \nabla^{\Omega}_X Y, \nu \rangle
$$ 
and 
$$
H=\mbox{tr}_g h.
$$
The principal curvatures of $\Sigma$ are defined to be the
eigenvalues of $h$.
Let $\mbox{Ric}_{\Omega}$ be the Ricci curvature tensor of $\Omega$.
Let $dA$ and $da$ be the canonical volume element of $\Omega$ and $\Sigma$ respectively.

Let $\phi$ be a smooth function on $\Omega$. Denote by $\mathbb{L}_{\phi}$ the weighted Laplacian operator on $\Omega$, while by $\mathbb{L}_{\phi}^{\Sigma}$ the weighted Laplacian operator on $\Sigma$, where
$$\mathbb{L}_{\phi}=\Delta-\langle \nabla \phi,\nabla\cdot \rangle$$
and
$$\mathbb{L}_{\phi}^{\Sigma}=\Delta_{\Sigma}-\langle \nabla_{\Sigma} \phi,\nabla_{\Sigma}\cdot \rangle.$$
Let $\mbox{Ric}^{\phi}_{\Omega}$ be the Bakry-Émery-Ricci tensor of $\Omega$, where 
$$\mbox{Ric}^{\phi}_{\Omega} = \mbox{Ric}_{\Omega} + \nabla^2 \phi.$$
We denote by $H_{\phi}$ the weighted mean curvature of $\Sigma$ with respect to $\nu$, where 
$$
H_{\phi}=H-\frac{\partial \phi}{\partial \nu}.
$$
Then $(\Omega^{n+1},g,e^{-\phi}dA)$ is often called a smooth metric measure space.
We refer interested readers to \cite{wei2009comparison} for more information about metric measure
spaces.

Now, we introduce the integral identities which will be used in our proofs. These formulas can be directly proven by standard calculations, readers can also refer to \cite{deng2020sharp}.

The first one is the following weighted Reilly-type formula for the weighted Laplacian operator. 

\begin{proposition}
Let $(\Omega^{n+1}, g)$ be an $(n+1)$-dimensional ($n\geq 2$) smooth compact connected Riemannian manifold with smooth boundary $\Sigma$, $V$ be a given a.e. twice differentiable function on $\Omega$ and $\phi$ be a smooth function on $\Omega$. Then for any smooth function $u$, we have 
\begin{equation}
\begin{aligned}
     &\int_{\Omega} V\bigg ((\mathbb{L}_{\phi} u)^2-|\nabla ^2 u|^2 \bigg) e^{-\phi}dA\\
           =&\int_{\Sigma} V\bigg [2(\mathbb{L}^{\Sigma}_{\phi} u) \frac{\partial u}{\partial \nu} + H_{\phi}\bigg(\frac{\partial u}{\partial \nu}\bigg)^2 + h(\nabla_{\Sigma}u, \nabla_{\Sigma}u)\bigg]e^{-\phi}da\\
           &+\int_{\Sigma} \frac{\partial V}{\partial \nu}|\nabla_{\Sigma} u|^2 e^{-\phi}da  \\
            &+\int_{\Omega}\bigg ( (\nabla^2 V - \mathbb{L}_{\phi} V g + V\mbox{Ric}^{\phi}_{\Omega})(\nabla u,\nabla u)\bigg)e^{-\phi}dA.
\label{f:2.1}
\end{aligned}
\end{equation}
\label{p:2.1}
\end{proposition}

\begin{remark}
We point out that when $V\equiv 1$, (\ref{f:2.1}) is the Reilly-type formula for the weighted Laplacian operator (see \cite{ma2010extension}); when $\phi=$ const,  (\ref{f:2.1}) is the weighted Reilly-type formula in \cite{qiu2015generalization}; when $V\equiv 1$ and $\phi=$ const, (\ref{f:2.1}) is just the classical Reilly's formula (see \cite{reilly1977applications}).
\end{remark}

The second one is the following Pohozaev-type identity for the weighted Laplacian operator. 

\begin{proposition}
Let $(\Omega^{n+1}, g)$ be an $(n+1)$-dimensional ($n\geq 2$) smooth compact connected Riemannian manifold with smooth boundary $\Sigma$, $X$ be a Lipschitz continuous vector field on $\Omega$ and $\phi$ be a smooth function on $\Omega$. Then for any smooth function $u$ which satisfies $\mathbb{L}_{\phi}u=0$, we have 
\begin{equation}
\begin{aligned}
           &\int_{\Omega} 
           \bigg (\langle \nabla^{\Omega}_{\nabla u}X,\nabla u \rangle -\frac{1}{2}|\nabla u|^2 \mbox{div}_{\phi}(X) \bigg) e^{-\phi}dA\\
           =&\int_{\Sigma} \bigg (\frac{\partial u}{\partial \nu}\langle X,\nabla u \rangle -\frac{1}{2}|\nabla u|^2 \langle X, \nu \rangle \bigg)e^{-\phi} da,
\end{aligned}
\label{f:2.2}
\end{equation}
where $\mbox{div}_{\phi}=\mbox{div}-\langle \nabla \phi,\cdot \rangle$ denotes the weighted divergence operator on $\Omega$.
\label{p:2.2}
\end{proposition}

\begin{remark}
We point out that when $\phi=$ const, (\ref{f:2.2}) is the Pohozaev-type identity in \cite{xiong2018comparison}.
\end{remark}

\section{The construction of the weight function and its Greene-Wu type approximation}
In this section, we concentrate on the construction of the weight function and its Greene-Wu type approximation. 

Let $(\Omega^{n+1}, g)$ be an $(n+1)$-dimensional ($n\geq 2$) smooth compact connected Riemannian manifold with smooth boundary $\Sigma$ and $\rho=d(\cdot,\Sigma)$ be the distance function to $\Sigma$. Let
$$\rho_{\max}=\max_{\Omega}\rho$$
and 
$f \in C([0,\rho_{\max}])$ be a continuous function. 
$\forall \, x \in \Omega \setminus \mbox{Cut}(\Sigma)$, we have defined the infimum of the radial sectional curvature $K(x)$.

In this section, we always assume that the Ricci curvature $\mbox{Ric}_{\Omega} \geq 0$, the infimum of the radial sectional curvature $K(x)\geq f(\rho(x))$, $\forall \, x \in \Omega \setminus \mbox{Cut}(\Sigma)$, and the second fundamental form $h \geq cg_{\Sigma} >0$. 
We will provide our weight function on the manifold $\Omega$ discussed above.

We first point out that based on the curvature assumptions, we have (see \cite{MR3404748})
\begin{equation}
    \rho_{\max} \leq \frac{1}{c}.
    \label{ineq:6}
\end{equation}
Define $F$ by 
$$F(t)=\min\{0,\min_{x\in [0,t]}f(x)\},$$
we then know that $F$ is nonpositive, nonincreasing and continuous on $[0,\rho_{\max}]$ and $K(x)\geq F(\rho(x))$, $\forall \, x \in \Omega \setminus \mbox{Cut}(\Sigma)$. 
Let $\theta \in C^2 ([0,\rho_{\max}])$ be the unique solution of the following system
\begin{equation*}
\begin{cases}
\begin{aligned}
    	&\theta^{\prime \prime}+F\theta=0,\\
    	&\theta(0)=1,\\
        &\theta^{\prime}(0)=-c.
    \end{aligned}
\end{cases}
\end{equation*}
We then have
\begin{proposition}
    $\forall \, t \in [0,\rho_{\max}]$, the function $\theta$ satisfies
    \begin{equation}
        \theta (t) \geq 1-ct.
        \label{ineq:7}
    \end{equation}
    \label{prop:3.1}
\end{proposition}
\begin{proof}
    Let $\widehat{\theta}=1-ct$, we then have
    \begin{equation*}
\begin{cases}
\begin{aligned}
    	&\widehat{\theta}^{\prime \prime}=0,\\
    	&\widehat{\theta}(0)=1,\\
        &\widehat{\theta}^{\prime}(0)=-c.
    \end{aligned}
\end{cases}
\end{equation*}
Suppose $\min \theta<0$, let $\theta_0 \in (0,\rho_{\max})$ be the first zero point of $\theta$, then on the interval $[0,\theta_0)$, we have
$$(\theta^{\prime}\widehat{\theta}-\widehat{\theta}^{\prime}\theta)^{\prime}=\theta^{\prime\prime}\widehat{\theta}-\widehat{\theta}^{\prime\prime}\theta=-F\theta\widehat{\theta}\geq0,$$
which shows
\begin{equation}
    \frac{\theta^{\prime}}{\theta}\geq \frac{\widehat{\theta}^{\prime}}{\widehat{\theta}}.
    \label{ineq:8}
\end{equation}
By the inequality (\ref{ineq:8}) and the initial conditions, we conclude that
\begin{equation*}
    \theta (t) \geq 1-ct,~\forall\,t\in [0,\theta_0].
\end{equation*}
We thus have $\theta_0 \geq \frac{1}{c}\geq \rho_{\max}$,
which is a contradiction. We then finish the proof.
\end{proof}

Let $V(t)=\int_{0}^{t} \theta(s)ds$ and we define the weight function (denote also by $V$) as
$$V(x)=V(\rho(x))=\int_{0}^{\rho(x)} \theta(s)ds.$$
It is easy for us to show that $V$ is Lipschitz continuous on $\Omega$ and $V \in C^3(\Omega \setminus \mbox{Cut}(\Sigma))$. 
By Proposition \ref{prop:3.1}, we also conclude that $V > 0$ in $\Omega^{\circ}$ and $V^{\prime}(\rho)\geq1-c\rho \geq0$. 
In addition, a direct calculation shows
$$V|_{\Sigma}=0$$
and
$$\frac{\partial V}{\partial \nu}|_{\Sigma} =-\theta(0)=-1.$$

For this weight function $V$, we also have
 
\begin{proposition}
Let $V$ be the weight function discussed above, then $\forall ~x \in \Omega$ and $X \in T_{x}\Omega$ with $|X|=1$, we have
\begin{equation}
      C(-V(\rho))(x;X) \geq c+ (\widehat{F}V)(x),
\label{ineq:9}
\end{equation}
where 
$$ C(-V(\rho))(x;X)=\liminf_{r \rightarrow 0} \frac{-V(\rho(\exp_{x}(rX))) -V(\rho(\exp_{x}(-rX))) +2V(\rho(x))}{r^2}$$
and $\widehat{F}=F\circ\rho$.
In particular, if $x \in \Omega \setminus \mbox{Cut}(\Sigma)$, we also have
$$\nabla^2 (-V)|_{x}(X,X)\geq c+ (\widehat{F}V)(x).$$
\label{prop:3.2}
\end{proposition}

\begin{proof}
$\forall ~x \in \Omega \setminus \mbox{Cut}(\Sigma)$, let $\gamma:[0,l] \rightarrow \Omega$ be the unique minimizing geodesic with arc length parameter such that $t = \rho(\gamma(t))$ for $t \in [0,l]$, where $\gamma(0)\in \Sigma$ and $\gamma(l)=x$.
    
Since $V^{\prime}(\rho)\geq1-c\rho \geq0$, we know that $-V$ is nonincreasing as a function of $\rho$.
Then by the curvature assumptions and the Hessian
comparison theorem for $\rho$ (see Theorem 2.31 in \cite{kasue1982laplacian}), we conclude that for any $X \in T_{x}\Omega$ with $|X|=1$,
\begin{equation*}
\begin{aligned}
         C(-V(\rho))(x;X)&\geq-V^{\prime \prime} (l)\langle \gamma^{\prime}(l) ,X \rangle^2-V^{\prime}(l) \frac{\theta^{\prime}(l)}{\theta(l)} (1-\langle \gamma^{\prime}(l) ,X \rangle^2)\\
         &=-\theta^{\prime}(l)\\
         &=c+\int_{0}^{l} F(s)\theta(s) ds\\
         &\geq c+F(l)V(l)\\
         &=c+(\widehat{F}V)(x).
\end{aligned}
\end{equation*}
Since $\rho$ is smooth on $\Omega \setminus \mbox{Cut}(\Sigma)$, we then know that  
$$\nabla^2 (-V)|_{x}(X,X)=C(-V(\rho))(x;X)\geq c+ (\widehat{F}V)(x).$$

$\forall~x \in \mbox{Cut}(\Sigma)$, let $\gamma:[0,l] \rightarrow \Omega$ be a minimizing geodesic with arc length parameter such that $t = \rho(\gamma(t))$ for $t \in [0,l]$, where $\gamma(0)\in \Sigma$ and $\gamma(l)=x$. We then know that $\forall~ t \in [0,l)$, the
infimum of the radial sectional curvature 
$$K(\gamma(t))\geq f(t) \geq
 F(t).$$
 Based on the continuity of curvature tensor, we know that $\forall~ Y \in T_{x}\Omega$ with $|Y|=1$ and $Y \perp \gamma^{\prime}(l)$, $$\mbox{Sec}_{\Omega}(\gamma^{\prime}(l)\wedge Y) \geq f(l) \geq
 F(l).$$
 The remaining part of the proof is the same as the proof for points in $\Omega \setminus \mbox{Cut}(\Sigma)$, so we omit it.
\end{proof}

We now consider the Greene-Wu type approximation of $V$. Let's first recall the definitions of $\xi$-convex function and $\eta$-convex function (see \cite{greene1979c}), where $\xi$ is a real number and $\eta$ is a continuous function.

\begin{defn}
Let $M$ be a Riemannian manifold, $\psi:M\rightarrow \mathbb{R}$ be a continuous function on $M$ and $\xi$ be a real number. We call $\psi$ a $\xi$-convex function at a point $p \in M$ if there exists a positive constant $\delta$ such that the function $\Psi(x)=\psi(x)-\frac{\xi+\delta}{2}d^2(p,x)$ is convex in a neighborhood of $p$. Let $\eta: M\rightarrow \mathbb{R}$ be a continuous function, then $\psi$ is called $\eta$-convex on $M$ if, for each $p \in M$, $\psi$ is $\eta(p)$-convex at $p$.
\label{d:3.3} 
\end{defn}

We choose three neighborhoods $O_1$, $O_2$ and $O_3$ of $\mbox{Cut}(\Sigma)$ such that
$$O_1 \subset \subset O_2 \subset \subset O_3 \subset \subset \Omega,$$
where “$A \subset \subset B$” for two sets $A$ and $B$ means “$\overline{A} \subset B$ and $\overline{A}$ is compact”. Then we have

\begin{proposition}
 $\forall \,\epsilon >0$, the function $-V$ is $(c+ \widehat{F}V-\epsilon)$-convex on $O_3$.
\label{prop:3.4}
\end{proposition}

\begin{proof}
    Fix an $\epsilon>0$ and a point $p \in O_3$, we only need to show that the function
    $-V$ is $(c+ (\widehat{F}V)(p)-\epsilon)$-convex at $p$.
    
    Let $\delta=\frac{\epsilon}{10}$ and $C= \max_{\Omega} (c+\widehat{F}V-\epsilon)$. Since $\widehat{F}V$ is continuous on $\Omega$, we then can choose a neighborhood $U_1$ of $p$ such that $\forall \, q \in U_1$, 
    $$|(\widehat{F}V)(p)-(\widehat{F}V)(q)|<\frac{\epsilon}{10}.$$
    In addition, since
    $$\nabla^2 (d^2(p,\cdot))|_{p}=2g_{p},$$
    we can choose a neighborhood $U_2$ of $p$ such that $\forall \, q \in U_2$,
    \begin{equation}
        2(1-\widehat{\epsilon})g_{q}\leq\nabla^2 (d^2(p,\cdot))|_{q}\leq 2(1+\widehat{\epsilon})g_{q},
        \label{ineq:3.4}
    \end{equation}
    where $\widehat{\epsilon}\in (0, \min\{1,\frac{\epsilon}{20(C+\delta)}\})$. Let $U=U_1 \cap U_2$, we now show that the function 
    $$\Psi(q)=-V(q)-\frac{c+(\widehat{F}V)(p)-\epsilon+\delta}{2}d^2(p,q)$$
    is convex on $U$. In fact, we only need to show that for any $q \in U$ and any unit $X \in T_{q}\Omega$, 
    $$\Psi(\exp_{q}(rX))+\Psi(\exp_{q}(-rX))-2\Psi(q)\geq0$$
    for small $r>0$.
    By Proposition \ref{prop:3.2}, we know that when $r$ is small enough,
    $$-V(\exp_{q}(rX))-V(\exp_{q}(-rX))+2V(q)\geq \bigg(c+(\widehat{F}V)(q)-\frac{\epsilon}{2}\bigg)r^2.$$
     We then have
    \begin{equation*}
    \begin{aligned}
         &\Psi(\exp_{q}(rX))+\Psi(\exp_{q}(-rX))-2\Psi(q)\\
         \geq&\bigg(c+(\widehat{F}V)(q)-\frac{\epsilon}{2}\bigg)r^2-\frac{c+(\widehat{F}V)(p)-\epsilon+\delta}{2}A(r),
    \end{aligned}
    \end{equation*}
where $A(r)=d^2(p,\exp_{q}(rX))+d^2(p,\exp_{q}(-rX))-2d^2(p,q)$.
Then by (\ref{ineq:3.4}), we know that 
$$2(1-2\widehat{\epsilon})r^2\leq A(r)\leq2(1+2\widehat{\epsilon})r^2$$
for small $r$. Now if $c+ (\widehat{F}V)(p)-\epsilon +\delta\geq 0$, we then have
\begin{equation*}
    \begin{aligned}
         &\Psi(\exp_{q}(rX))+\Psi(\exp_{q}(-rX))-2\Psi(q)\\
         \geq&\bigg(c+(\widehat{F}V)(q)-\frac{\epsilon}{2}\bigg)r^2-(c+(\widehat{F}V)(p)-\epsilon+\delta)(1+2\widehat{\epsilon})r^2\\
         =&\bigg(\frac{2\epsilon}{5}+(\widehat{F}V)(q)-(\widehat{F}V)(p)-2\widehat{\epsilon}(c+(\widehat{F}V)(p)-\epsilon+\delta)\bigg)r^2\\
         \geq&\frac{\epsilon}{5}r^2\\
         \geq&0.
    \end{aligned}
    \end{equation*}
Similarly, we can also show that $\Psi(\exp_{q}(rX))+\Psi(\exp_{q}(-rX))-2\Psi(q) \geq 0$ when $c+ (\widehat{F}V)(p)-\epsilon +\delta< 0$. So we finish the proof.
\end{proof}

We now consider the Riemannian convolution which is introduced by Greene-Wu (see \cite{greene1973subharmonicity,greene1976c,greene1979c}). In fact, we have 
$$\widehat{V}_{\tau}(x)=\frac{1}{\tau^{n+1}}\int_{v\in T_{x}\Omega} V(\exp_{x}v)k\bigg(\frac{|v|}{\tau}\bigg) d\mu_{x},$$
where $\mu_{x}$ is the Lebesgue measure on $T_x \Omega$ determined by the Riemannian metric $g$ at $x$ and $k: \mathbb{R}\rightarrow \mathbb{R}$ is a smooth nonnegative function which has its support contained in $[-1, 1]$, is a positive constant in a neighborhood of $0$ and satisfies
$$\int_{\mathbb{R}^{n+1}}k(|x|)dx=1.$$
To ensure that the function $\widehat{V}_{\tau}$ is well-defined on $O_3$, we always assume that $\tau < d(O_3,\Sigma)$. We then know that $\widehat{V}_{\tau}$ is a smooth function on $O_3$.
 
We now need the following approximation result (see \cite{greene1973subharmonicity,greene1976c,greene1979c}) for $\eta$-convex functions by its Riemannian convolution.

\begin{proposition}
Let $f$ be a $\eta$-convex function on a Riemannian manifold $M$ and $K$ is a compact subset of $M$, where $\eta:M\rightarrow \mathbb{R}$ is a continuous function.
Then there exist a neighborhood of $K$ and a $\tau_0 >0$ such that for all $\tau \in (0,\tau_0)$, the Riemannian convolution $\widehat{f}_{\tau}$ of $f$ is $\eta$-convex on the neighborhood.
\label{prop:3.5}
\end{proposition}

By Propositions \ref{prop:3.4} and \ref{prop:3.5}, we have

\begin{proposition}
 $\forall \,\epsilon >0$, there exists $\tau_0 >0$ such that for all $\tau \in (0,\tau_0)$, the function $-\widehat{V}_{\tau}$ is $(c+ \widehat{F}V-\epsilon)$-convex on $O_2$.
\label{prop:3.6}
\end{proposition}

Propositions \ref{prop:3.6} shows that $\forall \, \tau \in (0,\tau_0)$,
$$\nabla ^2 (-\widehat{V}_{\tau})|_{x}(X,X) \geq c+ (\widehat{F}V)(x)-\epsilon$$
for any $x \in O_2$ and any $X \in T_x\Omega$ with $|X|=1$. 

At the end of this section, by a gluing procedure, we finish the construction of the approximation of $V$.

Let $\phi$ be a smooth nonnegative cut-off function such that supp $\phi\subset O_2$ and $\phi \equiv1$ on $O_1$. The suitable approximation $V_{\tau}$ is defined by
\begin{equation}
    V_{\tau}=\phi \widehat{V}_{\tau}+(1-\phi)V.
    \label{e:10}
\end{equation}
It is easy to show that $V_{\tau} \in C^3(\Omega)$, $V_{\tau} \geq 0$ on $\Omega$ and $V=V_{\tau}$ on $\Omega \setminus O_2$. In addition, we have
\begin{equation}
    V_{\tau}-V=\phi(\widehat{V}_{\tau}-V).
    \label{e:11}
\end{equation}
So according to the properties of Riemannian convolution (see \cite{greene1973subharmonicity,greene1976c,greene1979c}) and (\ref{e:11}), we have
$$\lim_{\tau\rightarrow 0}||V_{\tau}-V||_{C^{0}(\Omega)}=0.$$
We now consider the Hessian of $-V_{\tau}$. In fact, we have the following

\begin{proposition}
 $\forall \,\epsilon >0$, there exists a $\tau(\epsilon) >0$ small enough such that 
 $$\nabla ^2 (-V_{\tau(\epsilon)}) \geq (c+ \widehat{F}V-\epsilon)g.$$
\label{prop:3.7}
\end{proposition}
\begin{proof}
    Fix an $\epsilon>0$.
    On $\Omega\setminus O_2$, we know that $V_{\tau}=V$ and by Proposition \ref{prop:3.2}, we have
    $$\nabla ^2 (-V_{\tau}) \geq (c+ \widehat{F}V)g.$$
    On $O_1$, we know that $V_{\tau}=\widehat{V}_{\tau}$ and by Proposition \ref{prop:3.6}, we know that 
    $$\nabla ^2 (-V_{\tau}) \geq (c+ \widehat{F}V-\epsilon)g$$
    when $\tau$ is small enough. On $\overline{O}_2\setminus O_1$, by the properties of Riemannian convolution (see \cite{greene1973subharmonicity,greene1976c,greene1979c}) and (\ref{e:11}), we know that
    $$\lim_{\tau\rightarrow 0}||V_{\tau}-V||_{C^{2}(\overline{O}_2\setminus O_1)}=0.$$
    So for the $\epsilon>0$, we know that there exists a $\tau(\epsilon)>0$ such that
    $$\nabla ^2 (-V_{\tau(\epsilon)})|_{x}(X,X) \geq \nabla ^2 (-V)|_{x}(X,X)-\epsilon \geq c+ \widehat{F}V(x)-\epsilon$$
    for any $x\in \overline{O}_2\setminus O_1$ and any $X \in T_x \Omega$ with $|X|=1$. We then finish the proof.
    
\end{proof}

All in all, we can provide the following proposition 
\begin{proposition}
 Let $\Omega$ and $V$ be the manifold and function discussed above. Fix a neighborhood $O$ of $\mbox{Cut}(\Sigma)$ in $\Omega$. Then $\forall \,\epsilon >0$, there exists a nonnegative function $V_{\epsilon} \in C^{3}(\Omega)$ such that 
 $V_{\epsilon}=V$ on $\Omega \setminus O$ and 
 $$\nabla ^2 (-V_{\epsilon}) \geq (c+ \widehat{F}V-\epsilon)g.$$
 In particular, we also have
$$ \lim_{\epsilon\rightarrow 0}||V_{\epsilon}-V||_{C^{0}(\Omega)}=0.$$
\label{prop:3.8}
\end{proposition}

\section{Proofs of Theorems \ref{thm: 1.1} and \ref{thm: 1.3}}

In this section, we concentrate on the proofs of Theorem \ref{thm: 1.1} and \ref{thm: 1.3}. We first prove Theorem \ref{thm: 1.1}.

Let $(\Omega^{n+1},g)$ be as in Theorem \ref{thm: 1.1}. By choosing $f\equiv K$, we know that $\theta$ satisfies
\begin{equation*}
\begin{cases}
\begin{aligned}
    	&\theta^{\prime \prime}+\overline{K}\theta=0,\\
    	&\theta(0)=1,\\
        &\theta^{\prime}(0)=-c,
    \end{aligned}
\end{cases}
\end{equation*}
where $\overline{K}=\min\{0,K\}$.
We then conclude that
$$
    \theta=
    \left\{
    \begin{aligned}
    	&1-ct,~~&K\geq0,\\
    	&\cosh(\sqrt{|K|}t)-\frac{c}{\sqrt{|K|}}\sinh(\sqrt{|K|}t),~~& K < 0,
    \end{aligned}
    \right.
    $$
and 
$$
    V=
    \left\{
    \begin{aligned}
    	&\rho-\frac{c}{2}\rho^2,~~&K\geq0,\\
    	&\frac{1}{\sqrt{|K|}}\sinh(\sqrt{|K|}\rho)-\frac{c}{|K|}\cosh(\sqrt{|K|}\rho)+\frac{c}{|K|},~~& K < 0. 
    \end{aligned}
    \right.
    $$
In addition, by Proposition \ref{prop:3.8} and the fact that $V \in C^{\infty}(\Omega\setminus \mbox{Cut}(\Sigma))$, we have
\begin{proposition}
 Let $(\Omega^{n+1},g)$ be as in Theorem \ref{thm: 1.1} and $V$ be the weight function discussed above. Fix a neighborhood $O$ of $\mbox{Cut}(\Sigma)$ in $\Omega$. Then $\forall \,\epsilon >0$, there exists a nonnegative function $V_{\epsilon} \in C^{\infty}(\Omega)$ such that 
 $V_{\epsilon}=V$ on $\Omega \setminus O$ and 
$$
    \nabla ^2 (-V_{\epsilon})\geq(c+\overline{K}V-\epsilon)g. 
    $$
 In particular, we also have
$$ \lim_{\epsilon\rightarrow 0}||V_{\epsilon}-V||_{C^{0}(\Omega)}=0.$$
\label{prop:4.1}
\end{proposition}

We then can provide a key inequality for the proof of Theorem \ref{thm: 1.1}.
\begin{proposition}
Let $(\Omega^{n+1},g)$ be as in Theorem \ref{thm: 1.1} and $u$ be a harmonic function on $\Omega$. Then we have
\begin{equation}
     \int_{\Sigma} \bigg(\frac{\partial u}{ \partial \nu}\bigg )^2da  \geq \int_{\Omega} (c+\overline{K}V)|\nabla u|^2 dA.
\label{f:4.1}
\end{equation}
\label{prop:4.2}
\end{proposition}

\begin{proof}
By the construction of $V_{\epsilon}$, we know that
$$V_{\epsilon}|_{\Sigma}=0$$
and
$$\frac{\partial V_{\epsilon}}{\partial \nu}|_{\Sigma}=-1.$$
Then by the weighted Reilly-type formula in Proposition \ref{p:2.1} (let $\phi \equiv 0$), we have
\begin{equation}
\begin{aligned}
           -&\int_{\Omega} V_{\epsilon}|\nabla ^2 u|^2  dA    
           =-\int_{\Sigma} |\nabla_{\Sigma} u|^2 da \\
          & +\int_{\Omega} (\nabla^2 V_{\epsilon} - \Delta V_{\epsilon} g + V_{\epsilon}\mbox{Ric}_{\Omega})(\nabla u,\nabla u)dA.
\label{f:4.2}
\end{aligned}
\end{equation}
In addition, by the Pohozaev-type identity in Proposition \ref{p:2.2} (let $\phi \equiv 0$ and $X= \nabla V_{\epsilon}$), we have
\begin{equation}
\begin{aligned}
           \int_{\Omega} 
            (2\nabla^2 V_{\epsilon} -\Delta V_{\epsilon}g) (\nabla u,\nabla u) dA
           =\int_{\Sigma}\bigg ( |\nabla_{\Sigma} u|^2- \bigg(\frac{\partial u}{\partial \nu}\bigg)^2 \bigg)  da.
\end{aligned}
\label{f:4.3}
\end{equation}
Then by (\ref{f:4.2}) and (\ref{f:4.3}), we have
\begin{equation}
\begin{aligned}
           \int_{\Sigma} \bigg(\frac{\partial u}{ \partial \nu}\bigg )^2   da = \int_{\Omega} 
           \bigg (-\nabla^2 V_{\epsilon}(\nabla u,\nabla u)+V_{\epsilon}|\nabla ^2 u|^2 +  V_{\epsilon}\mbox{Ric}_{\Omega}(\nabla u,\nabla u)\bigg) dA.
\end{aligned}
\label{f:4.4}
\end{equation}
By Proposition \ref{prop:4.1}, the curvature assumptions in Theorem \ref{thm: 1.1} and letting $\epsilon \rightarrow 0$, we know that
\begin{equation*}
     \int_{\Sigma} \bigg(\frac{\partial u}{ \partial \nu}\bigg )^2da  \geq \int_{\Omega} (c+\overline{K}V)|\nabla u|^2 dA.
\end{equation*}
\end{proof}

\begin{proof1}
    Let $u$ be an eigenfunction corresponding to the first nonzero Steklov eigenvalue $\sigma_1$. We then have
    \begin{equation}
        \int_{\Sigma} \bigg(\frac{\partial u}{ \partial \nu}\bigg )^2da = \sigma_1 ^2 \int_{\Sigma} u^2da
        \label{equa:1.1}
    \end{equation}
and
\begin{equation}
    \int_{\Omega} |\nabla u|^2 dA= \sigma_1 \int_{\Sigma} u^2da.
    \label{equa:1.2}
\end{equation}
Note that $-V$ is nonincreasing as a function of $\rho$ and $\rho_{\max}\leq\frac{1}{c}$, we conclude that for any $x \in \Omega$,
$$\overline{K}V(\rho(x))\geq\overline{K}V\bigg(\frac{1}{c}\bigg).$$
Then by (\ref{f:4.1}), (\ref{equa:1.1}) and (\ref{equa:1.2}), we have
$$\sigma_1\geq c+\overline{K}V\bigg(\frac{1}{c}\bigg)=\left\{
    \begin{aligned}
    	&c,~~&K\geq0,\\
    	&c\cosh{\bigg(\frac{\sqrt{|K|}}{c}\bigg)}-\sqrt{|K|}\sinh{\bigg(\frac{\sqrt{|K|}}{c}\bigg)},~~&K \leq 0. 
    \end{aligned}
    \right.$$
Combining Theorem 8 in \cite{escobar1997geometry} and the remark of Theorem \ref{thm: 1.1}, we conclude that
$$\sigma_1\geq\sigma(c,K),$$
where 
$$
    \sigma(c,K)=
    \left\{
    \begin{aligned}
    	&c,~~&K\geq0,\\
    	&c\cosh{\bigg(\frac{\sqrt{|K|}}{c}\bigg)}-\sqrt{|K|}\sinh{\bigg(\frac{\sqrt{|K|}}{c}\bigg)},~~&K_0\leq K \leq 0, \\
    	&\frac{c}{2},~~~&K\leq K_0,
    \end{aligned}
    \right.
    $$
and 
$K_0 \in (-\infty,0)$ is the unique number which satisfies 
$$c\cosh{\bigg(\frac{\sqrt{|K_0|}}{c}\bigg)}-\sqrt{|K_0|}\sinh{\bigg(\frac{\sqrt{|K_0|}}{c}\bigg)}=\frac{c}{2}.$$ 

When $\Omega$ is isometric to a Euclidean ball of radius $\frac{1}{c}$, we know that $K=0$ and $\sigma_1=\sigma(c,K)=c.$

Now we assume that $\sigma_1 = \sigma(c,K)$. Since $\sigma_1 >\frac{c}{2}$, we know that $K \in (K_0,\infty)$. By (\ref{f:4.4}), we have
\begin{equation}
\begin{aligned}
      \sigma(c,K)\int_{\Omega} |\nabla u|^2dA
      &=\int_{\Sigma} \bigg(\frac{\partial u}{ \partial \nu}\bigg )^2 da \\
      &=\int_{\Omega} \bigg (-\nabla^2 V_{\epsilon}(\nabla u,\nabla u)+V_{\epsilon}|\nabla ^2 u|^2 +  V_{\epsilon}\mbox{Ric}_{\Omega}(\nabla u,\nabla u)\bigg) dA\\
      &\geq\int_{\Omega} 
    \bigg ((\sigma(c,K)-\epsilon)|\nabla u|^2+V_{\epsilon}|\nabla ^2 u|^2 +  V_{\epsilon}\mbox{Ric}_{\Omega}(\nabla u,\nabla u)\bigg) dA.
\end{aligned}
\label{f:4.5}
\end{equation}
Then by letting $\epsilon \rightarrow 0$, we have
$$\int_{\Omega}\bigg (V|\nabla ^2 u|^2+V\mbox{Ric}_{\Omega}(\nabla u,\nabla u)\bigg) dA=0.$$
We then have the following Obata equation
\begin{equation}
    \begin{cases}
\begin{aligned}
&\nabla^2 u=0,& {\text{in}} \,\Omega,\\
&\frac{\partial u}{\partial \nu}=\sigma(c,K)u,&{\text{on}} \, \Sigma.
\end{aligned}
\end{cases}
\end{equation}
Then by Proposition 4.3 in \cite{xia2023escobar} and the fact that $c\geq\sigma(c,K)$, we conclude that $\Omega$ is isometric to a Euclidean ball of radius $\frac{1}{\sigma(c,K)}$, which shows $K=0$ and $\sigma_1=\sigma(c,K)=c$.
\end{proof1}

\begin{proof4}
    We still use the weight function $V$ and its Greene-Wu type approximation $V_{\epsilon}$ in the proof of Theorem \ref{thm: 1.1}. Since $K \geq K_0$, we know that $\sigma(c,K)=c+\overline{K}V(\frac{1}{c})$. By (\ref{f:4.4}) and the fact that $f$ is harmonic, we have
\begin{equation*}
\begin{aligned}
           \int_{\Sigma} (kf)^2 da&=
           \int_{\Sigma} \bigg(\frac{\partial f}{ \partial \nu}\bigg )^2   da\\
           &= \int_{\Omega} 
           \bigg (-\nabla^2 V_{\epsilon}(\nabla f,\nabla f)+V_{\epsilon}|\nabla ^2 f|^2 +  V_{\epsilon}\mbox{Ric}_{\Omega}(\nabla f,\nabla f)\bigg) dA\\
           &\geq\bigg(c+\overline{K}V\bigg(\frac{1}{c}\bigg)-\epsilon\bigg)\int_{\Sigma}kf^2da+\int_{\Omega}V_{\epsilon}|\nabla ^2 f|^2 dA.
\end{aligned}
\end{equation*}
Then by letting $\epsilon \rightarrow 0$, we have
$$0=\int_{\Sigma}\bigg(c+\overline{K}V\bigg(\frac{1}{c}\bigg)-k\bigg)kf^2da+\int_{\Omega}V|\nabla ^2 f|^2 dA,$$
which shows $k=\sigma(c,K)$ when $f \not=0$ and $\nabla^2 f =0$. Without loss of generality, we can assume that $|\nabla f|^2=1$ and we then have
$|\nabla_{\Sigma}f|^2+k^2f^2=1$, which shows that $f^{-1}(0)\cap \Sigma$ is an $(n-1)$-dimensional submanifold of $\Sigma$. Thus by the continuity of $k$, we know that 
$$k\equiv\sigma(c,K).$$
Then we conclude that 
$f$ be an eigenfunction corresponding to the first nonzero Steklov eigenvalue $\sigma_1$ and $\sigma_1=\sigma(c,K)$. Then by Theorem \ref{thm: 1.1}, we know that $\Omega$ is isometric to a Euclidean ball of radius $\frac{1}{c}$ and $k \equiv c$.
\end{proof4}

We now present the proof of Theorem \ref{thm: 1.3}. Similar to Proposition \ref{prop:4.2}, we have 
\begin{proposition}
Let $(\Omega^{n+1},g)$ be as in Theorem \ref{thm: 1.3} and $u$ be a harmonic function on $\Omega$. Then we have
\begin{equation}
     \int_{\Sigma} \bigg(\frac{\partial u}{ \partial \nu}\bigg )^2da  \geq c\int_{\Omega} |\nabla u|^2 dA.
\label{f:4.6}
\end{equation}
\label{prop:4.3}
\end{proposition}
\begin{proof}
Let $V$ be the weight function discussed in section 3 and $V_{\epsilon}$ be the Greene-Wu type approximation of $V$ in Proposition \ref{prop:3.8}. By a similar proof of Proposition \ref{prop:4.2}, we have
\begin{equation}
\begin{aligned}
           \int_{\Sigma} \bigg(\frac{\partial u}{ \partial \nu}\bigg )^2   da = \int_{\Omega} 
           \bigg (-\nabla^2 V_{\epsilon}(\nabla u,\nabla u)+V_{\epsilon}|\nabla ^2 u|^2 +  V_{\epsilon}\mbox{Ric}_{\Omega}(\nabla u,\nabla u)\bigg) dA.
\end{aligned}
\label{f:4.9}
\end{equation}
By Proposition \ref{prop:3.8}, the curvature assumptions in Theorem \ref{thm: 1.3} and letting $\epsilon \rightarrow 0$, we have
\begin{equation*}
     \int_{\Sigma} \bigg(\frac{\partial u}{ \partial \nu}\bigg )^2da  \geq \int_{\Omega} (c+(\widehat{F}+|\widehat{F}|)V)|\nabla u|^2 dA \geq c \int_{\Omega} |\nabla u|^2 dA.
\end{equation*}
\end{proof}

\begin{proof2}
      Let $u$ be an eigenfunction corresponding to the first nonzero Steklov eigenvalue $\sigma_1$. We then have
$$\int_{\Sigma} \bigg(\frac{\partial u}{ \partial \nu}\bigg )^2da = \sigma_1 ^2 \int_{\Sigma} u^2da$$
and
$$\int_{\Omega} |\nabla u|^2 dA= \sigma_1 \int_{\Sigma} u^2da.$$
Then by (\ref{f:4.6}), we conclude that
$$\sigma_1\geq c.$$

When $\Omega$ is isometric to a Euclidean ball of radius $\frac{1}{c}$, we know that $f\equiv0$ and $\sigma_1=c$.

Now we assume that $\sigma_1 = c$. By (\ref{f:4.9}), we have
\begin{equation*}
\begin{aligned}
      c\int_{\Omega} |\nabla u|^2dA
      &=\int_{\Sigma} \bigg(\frac{\partial u}{ \partial \nu}\bigg )^2 da \\
      &=\int_{\Omega} \bigg (-\nabla^2 V_{\epsilon}(\nabla u,\nabla u)+V_{\epsilon}|\nabla ^2 u|^2 +  V_{\epsilon}\mbox{Ric}_{\Omega}(\nabla u,\nabla u)\bigg) dA\\
      &\geq\int_{\Omega} 
    \bigg ((c+\widehat{F}V+|\widehat{F}|V_{\epsilon}-\epsilon)|\nabla u|^2+V_{\epsilon}|\nabla ^2 u|^2 \bigg) dA.
\end{aligned}
\end{equation*}
Then by letting $\epsilon \rightarrow 0$, we have
$$\int_{\Omega}V|\nabla ^2 u|^2 dA=0.$$
We then have the following Obata equation
\begin{equation}
    \begin{cases}
\begin{aligned}
&\nabla^2 u=0,& {\text{in}} \,\Omega,\\
&\frac{\partial u}{\partial \nu}=cu,&{\text{on}} \, \Sigma.
\end{aligned}
\end{cases}
\end{equation}
Then by Proposition 4.3 in \cite{xia2023escobar}, we conclude that $\Omega$ is isometric to a Euclidean ball of radius $\frac{1}{c}$, which shows $\widehat{F}\equiv0$ and $f\equiv0$.
\end{proof2}

\section{Proof of Theorem \ref{thm: 1.5}}

In this section, we prove Theorem \ref{thm: 1.5}. We first provide the relationship between the Laplacian operator with respect to the conformal metric and the weighted Laplacian operator.

\begin{lemma}
    Let $(M^{n+1},g)$ be a Riemannian manifold and $\widehat{g}=e^{2f}g$, where $f\in C^{\infty}(M)$. Then
    \begin{equation}
    \Delta_{\widehat{g}}=e^{-2f}\mathbb{L}_{-(n-1)f},
    \label{equa:5.1}
    \end{equation}
    where $\Delta_{\widehat{g}}$ is the Laplacian operator with respect to the metric $\widehat{g}$.
\label{l:5.1} 
\end{lemma}

\begin{proof}
    Let $(x_1,x_2,\cdots,x_{n+1})$ be a coordinate in $M$ and the metric has the form
   $$
    g=g_{ij}dx^i dx^j.
   $$
   Then we have 
   $$\widehat{g}=\widehat{g}_{ij}dx^i dx^j=e^{2f}g_{ij}dx^i dx^j.$$
   For any smooth function $u \in C^{\infty}(M)$, we have
   $$\begin{aligned}
      \Delta_{\widehat{g}}u
      &=\frac{1}{\sqrt{\widehat{g}}}\partial_{x_i}(\sqrt{\widehat{g}}\widehat{g}^{ij}\partial_{x_j}u)\\
      &=\frac{1}{e^{(n+1)f}\sqrt{g}}\partial_{x_i}(e^{(n-1)f}\sqrt{g}g^{ij}\partial_{x_j}u)\\
      &=e^{-2f}(\Delta u+(n-1)\langle\nabla f,\nabla u\rangle)\\
      &=e^{-2f}\mathbb{L}_{-(n-1)f}u.
\end{aligned}
$$
\end{proof}

Combining Lemma \ref{l:5.1} with Propositions \ref{p:2.1} and \ref{p:2.2}, we have the following two integral identities related to the Laplacian operator with respect to the conformal metric.

\begin{proposition}
Let $(\Omega^{n+1}, g)$ be an $(n+1)$-dimensional ($n\geq 2$) smooth compact connected Riemannian manifold with smooth boundary $\Sigma$ and $V$ be a given a.e. twice differentiable function on $\Omega$. Let $\widehat{g}=e^{2f}g$, where $f\in C^{\infty}(\Omega)$ is a smooth function on $\Omega$. Then for any smooth function $u$, we have 
\begin{equation}
\begin{aligned}
     &\int_{\Omega} V\bigg ((e^{2f}\Delta_{\widehat{g}}u)^2-|\nabla ^2 u|^2 \bigg) e^{(n-1)f}dA\\
           =&\int_{\Sigma} V\bigg (2L_{-(n-1)f}^{\Sigma} u \frac{\partial u}{\partial \nu} + H_{-(n-1)f}\bigg(\frac{\partial u}{\partial \nu}\bigg)^2 + h(\nabla_{\Sigma}u, \nabla_{\Sigma}u)\bigg)e^{(n-1)f}da\\
           &+\int_{\Sigma} \frac{\partial V}{\partial \nu}|\nabla_{\Sigma} u|^2 e^{(n-1)f}da  \\
            &+\int_{\Omega}\bigg ( (\nabla^2 V - e^{2f}\Delta_{\widehat{g}} V g + V\mbox{Ric}^{-(n-1)f}_{\Omega})(\nabla u,\nabla u)\bigg)e^{(n-1)f}dA.
\label{f:5.2}
\end{aligned}
\end{equation}
\label{prop:5.2}
\end{proposition}

\begin{proposition}
Let $(\Omega^{n+1}, g)$ be an $(n+1)$-dimensional ($n\geq 2$) smooth compact connected Riemannian manifold with smooth boundary $\Sigma$ and $V\in C^{2}(\Omega)$ be a twice continuously differentiable function on $\Omega$. Let $\widehat{g}=e^{2f}g$, where $f\in C^{\infty}(\Omega)$ is a smooth function on $\Omega$. Then for any smooth function $u$ which is harmonic with respect to the metric $\widehat{g}$, we have 
\begin{equation}
\begin{aligned}
           &\int_{\Omega} 
           \bigg (\nabla^2 V(\nabla u,\nabla u) -\frac{1}{2}|\nabla u|^2 e^{2f}\Delta_{\widehat{g}}V \bigg) e^{(n-1)f}dA\\
           =&\int_{\Sigma} \bigg (\frac{\partial u}{\partial \nu}\langle \nabla V,\nabla u \rangle -\frac{1}{2}\frac{\partial V}{\partial \nu}|\nabla u|^2 \bigg)e^{(n-1)f} da.
\end{aligned}
\label{f:5.3}
\end{equation}
\label{prop:5.3}
\end{proposition}

\begin{proof3}
     We still use the weight function $V$ and its Greene-Wu type approximation $V_{\epsilon}$ in Proposition \ref{prop:4.1}.
     Let $u$ be an eigenfunction corresponding to the first nonzero Steklov eigenvalue $\sigma_1(\widehat{g})$. Since $f|_{\Sigma}=0$, we know that
    \begin{equation*}
    \begin{cases}
    \begin{aligned}
    	&\Delta_{\widehat{g}} u=0, & \mbox{in} \,\Omega,\\
    	&\frac{\partial u}{\partial \nu}=\sigma_1(\widehat{g}) u, & \mbox{on} \, \Sigma.
    \end{aligned}
    \end{cases}
    \end{equation*}
Then by Propositions \ref{prop:5.2} and \ref{prop:5.3}, we have 
\begin{equation}
\begin{aligned}
     &-\int_{\Omega} V_{\epsilon}|\nabla ^2 u|^2 e^{(n-1)f}dA\\
     =&-\int_{\Sigma} |\nabla_{\Sigma} u|^2 da  \\
            &+\int_{\Omega}\bigg ( (\nabla^2 V_{\epsilon} - e^{2f}\Delta_{\widehat{g}} V_{\epsilon} g + V_{\epsilon}\mbox{Ric}^{-(n-1)f}_{\Omega})(\nabla u,\nabla u)\bigg)e^{(n-1)f}dA.
\label{f:5.4}
\end{aligned}
\end{equation}
and
\begin{equation}
\begin{aligned}
           &\int_{\Omega} 
           \bigg (\nabla^2 V_{\epsilon}(\nabla u,\nabla u) -\frac{1}{2}|\nabla u|^2 e^{2f}\Delta_{\widehat{g}}V_{\epsilon} \bigg) e^{(n-1)f}dA\\
           =&\frac{1}{2}\int_{\Sigma} \bigg (|\nabla_{\Sigma}u|^2-\bigg(\frac{\partial u}{\partial \nu}\bigg)^2 \bigg) da.
\end{aligned}
\label{f:5.5}
\end{equation}
Combining (\ref{f:5.4}) and (\ref{f:5.5}), we have 
\begin{equation}
\begin{aligned}
           &\int_{\Sigma} \bigg(\frac{\partial u}{\partial \nu}\bigg)^2  da \\
           \geq &\int_{\Omega} \bigg ( V_{\epsilon}|\nabla ^2 u|^2+(-\nabla^2 V_{\epsilon} + V_{\epsilon}\mbox{Ric}^{-(n-1)f}_{\Omega})(\nabla u,\nabla u)\bigg)e^{(n-1)f}dA\\
           \geq& \int_{\Omega} \bigg ( (-\nabla^2 V_{\epsilon} + V_{\epsilon}\mbox{Ric}_{\Omega}-(n-1)V_{\epsilon}\nabla^{2}f)(\nabla u,\nabla u)\bigg)e^{(n-1)f}dA.\\
\end{aligned}
\label{f:5.6}
\end{equation}
By Proposition \ref{prop:4.1}, the assumptions in Theorem \ref{thm: 1.5} and letting $\epsilon \rightarrow 0$, we have
$$\sigma_1(\widehat{g})\int_{\Sigma} u\frac{\partial u}{\partial \nu} da \geq c\int_{\Omega} |\nabla u|^2e^{(n-1)f}dA.$$
By (\ref{equa:5.1}) and the divergence theorem with respect to the weighted Laplacian operator, we know that
$$\int_{\Sigma} u\frac{\partial u}{\partial \nu} da = \int_{\Omega} |\nabla u|^2e^{(n-1)f}dA.$$
Thus we conclude that $\sigma_1(\widehat{g})\geq c$.

When $(\Omega,g)$ is isometric to a Euclidean ball of radius $\frac{1}{c}$ and $f\equiv0$, we know that $\sigma_1(\widehat{g})=c.$

Now we assume that $\sigma_1(\widehat{g})=c$. By a similar calculation in the proof of Theorem \ref{thm: 1.1}, we have
\begin{equation}
    \nabla^2 u=0, ~ \mbox{in} \,\Omega
    \label{f:5.7}
\end{equation}
and
\begin{equation}
    \nabla^2 f(\nabla u, \nabla u)=\frac{1}{n-1}\min\{0,K\}|\nabla u|^2, ~ \mbox{in} \,\Omega.
    \label{f:5.8}
\end{equation}
Since $\frac{\partial u}{\partial \nu}=cu$, by Proposition 4.3 in \cite{xia2023escobar}, we conclude that $(\Omega,g)$ is isometric to a Euclidean ball of radius $\frac{1}{c}$ which shows $K=0$, and $u$ is an eigenfunction corresponding to the first nonzero Steklov eigenvalue $c$ under the standard Euclidean metric $g_{euc}$. For convenience, we assume that 
$$\Omega=\bigg\{(x_1,x_2,\cdots,x_{n+1})\in \mathbb{R}^{n+1}~|~ \sum_{i=1}^{n+1} (x_i)^2\leq\frac{1}{c^2}\bigg\},$$
and 
$$u=x_1.$$
Then by (\ref{f:5.8}) and the assumption that $f|_{\Sigma}=0$, we know that
\begin{equation*}
    \begin{cases}
    \begin{aligned}
    	&\frac{\partial^2f}{\partial x_1^2}=0, & \mbox{in} \,\Omega,\\
    	&f=0, & \mbox{on} \, \Sigma,
    \end{aligned}
    \end{cases}
    \end{equation*}
which shows $f\equiv 0$ on $\Omega$.
\end{proof3}

\bibliographystyle{plain}
\bibliography{ref}

\vskip .5cm
\noindent
Yiwei Liu:\\
School of Mathematical Sciences, Shanghai Jiao Tong University\\
email: lyw201611012@sjtu.edu.cn

\vskip .5cm
\noindent
Yi-Hu Yang:\\
School of Mathematical Sciences, Shanghai Jiao Tong University\\
email: yangyihu@sjtu.edu.cn
\end{document}